\documentclass[11pt]{amsart}
\usepackage{amsmath,amsfonts,amsthm,amssymb,amscd}
\usepackage{float,graphicx}
\usepackage{hyperref}
\usepackage{upgreek}
\usepackage{mathtools}
\usepackage{xcolor}
\usepackage{diagbox} 
\graphicspath{{images/}}
\usepackage[letterpaper,top=3cm,bottom=3cm,left=3cm,right=3cm,marginparwidth=1.75cm]{geometry}
\usepackage{comment}

\newtheorem{theorem}{Theorem}[section]
\newtheorem{lemma}[theorem]{Lemma}
\newtheorem{proposition}[theorem]{Proposition}

\theoremstyle{remark}
\newtheorem{remark}[theorem]{Remark}
\numberwithin{equation}{section}

\newcommand{\zm}[1]{\textcolor{black}{#1}}

\title{Blowup analysis for a quasi-exact 1D model of 3D Euler and Navier-Stokes}
\author{ Thomas Y. Hou and Yixuan Wang}
 \address{Applied and Computational Mathematics, Caltech, Pasadena, CA 91125. \newline Emails: hou@cms.caltech.edu, roywang@caltech.edu}

\date{\today}

\begin{document}

\maketitle
\begin{abstract}
        We study the singularity formation of a quasi-exact 1D model proposed by Hou-Li in \cite{hou2008dynamic}. This model is based on an approximation of the axisymmetric Navier-Stokes equations in the $r$ direction. The solution of the 1D model can be used to construct an exact solution of the original 3D Euler and Navier-Stokes equations if the initial angular velocity, angular vorticity, and angular stream function are linear in $r$. This model shares many intrinsic properties similar to those of the 3D Euler and Navier-Stokes equations. It captures the competition between advection and vortex stretching as in the 1D De Gregorio \cite{de1990one, de1996partial} model. We show that the inviscid model with weakened advection and smooth initial data or the original 1D model with H\"older continuous data develops a self-similar blowup. We also show that the viscous model with weakened advection and smooth initial data develops a finite time blowup. 
To obtain sharp estimates for the nonlocal terms, we perform an exact computation  for the low-frequency Fourier modes and  {extract damping in leading order estimates}  for the high-frequency modes using singularly weighted norms {in the energy estimates}. The analysis for the viscous case is more subtle since the viscous terms produce some instability if we just use singular weights. We establish the blowup analysis for the viscous model by carefully designing an energy norm that combines a singularly weighted energy norm and a {sum of} high-order Sobolev norms.
    \end{abstract}
\section{Introduction}

Whether the 3D incompressible Euler and Navier-Stokes equations can develop a finite time singularity from smooth initial data is one of the most outstanding open questions in nonlinear partial differential equations. An essential difficulty is that the vortex stretching term has a quadratic nonlinearity in terms of vorticity. A simplified 1D model was proposed by the Constantin-Lax-Majda model (CLM model for short) \cite{constantin1985simple} to capture the effect of nonlocal vortex stretching. The CLM model can be solved explicitly and can develop a finite time singularity from smooth initial data. Later on, De Gregorio incorporated the advection term into the CLM model to study the competition between advection and vortex stretching \cite{de1990one,de1996partial}, see  \zm{\cite{constantin1986} for singularity formulation in the distorted Euler equations with transport neglected and also   \cite{hou2008dynamic,lei2009stabilizing} for a related study on the stabilizing effect of advection for the 3D Euler and Navier-Stokes equations. There have been recent studies on the effect of advection and vortex stretching in other related models; see \cite{sarria2013blow} for the generalized inviscid Proudman-Johnson equation, \cite{elgindi2022strong} with a Riesz transform added to the vorticity formulation of 2D Euler equation, and \cite{miller2023finite} with advection term dropped in the vorticity formulation of 3D Euler equation.}  In \cite{okamoto2008generalization},
Okamoto, Sakajo, and Wunsch further introduced a parameter for the advection term to measure the relative strength of the advection in the De Gregorio model. These simplified 1D models have inspired  many subsequent studies. Interested readers may consult the excellent surveys \cite{constantin2007euler,gibbon2008three,kiselev2018,majda2002vorticity} and the references therein. {Very recently, the authors in \cite{huang2023self} established self-similar blowup for the whole family of gCLM models with $a\leq1$ using a fixed-point argument.}
On the other hand, these 1D scalar models are phenomenological in nature and cannot be used to recover the solution of the original 3D Euler equations. 

{For the line of research on the singularity formulation for the 3D Euler equations,}  Luo-Hou \cite{luo2013potentially-2} presented in 2014 convincing numerical evidence that the 3D axisymmetric Euler equations with smooth initial data and boundary develop a potential finite time singularity. 
Inspired by Elgindi's recent breakthrough for finite time singularity of the axisymmetric Euler with no swirl and $C^{1,\alpha}$ velocity \cite{ elgindi2021finite}, Chen and Hou proved the finite time blowup of the 2D Boussinesq and 3D Euler equations with $C^{1,\alpha}$ initial velocity and boundary \cite{chen2019finite2}. \zm{For other recent works on singularity formulation of 3D Euler with limited regularity, see also \cite{chen2023remarks,cordoba2023finite} for initial data that is smooth except at the origin, \cite{cordoba2023blow} for more smooth data but with a 
 $C^{1/2-\epsilon}$ force, and \cite{elgindi2021incompressible, elgindi2019finite} for settings with nonsmooth boundary.} Very recently, Chen and Hou proved stable and nearly self-similar blowup of the 2D Boussinesq and 3D Euler with smooth initial data and boundary using computer assistance \cite{chen2022stable}. 

In 2008, Hou and Li \cite{hou2008dynamic} proposed a new 1D model for the 3D axisymmetric Euler and Navier-Stokes equations. This model approximates the 3D axisymmetric Euler and Navier-Stokes equations along the symmetry axis based on an approximation in the $r$ direction. The solution of the 1D model can be used to construct an exact solution of the original 3D Navier-Stokes equations if the initial angular velocity, angular vorticity, and angular stream function are linear in $r$. This model shares many intrinsic properties similar to those of the 3D Navier-Stokes equations. Thus, it captures some essential nonlinear features of the 3D Euler and Navier-Stokes equations.
In the same paper  \cite{hou2008dynamic}, the authors proved the global regularity of the Hou-Li model by deriving a new Lyapunov functional, which captures the exact cancellation between advection and vortex stretching. 

The purpose of this paper is to study the singularity formation of a weak advection version of the Hou-Li model for smooth data. We introduce a parameter $a$ to characterize the relative strength between advection and vortex stretching, just like the gCLM model. Both inviscid and viscous cases are considered. We also prove the finite time singularity formation of the original inviscid Hou-Li model ($a=1$ and $\nu=0$) with $C^\alpha$ initial data. Inspired by the recent work of Chen \cite{chen2021regularity} for the De Gregorio model, we consider the case of $a < 1$ and treat $1-a$ as a small parameter. For the $C^\alpha$ initial data, we consider the original Hou-Li model with $a=1$ and $1-\alpha$ small. 
By using the dynamic rescaling formulation and analyzing the stability of the linearized operator around an approximate steady state of the original Hou-Li model ($a=1$), we prove finite time self-similar blowup. 

We follow a general strategy that we have established in our previous works \cite{chen2021finite,chen2019finite2}.
Establishing linear stability of the approximate steady state is the most crucial step in our blowup analysis. To obtain sharp estimates for the nonlocal terms, we carry out an exact computation for the low-frequency Fourier modes and {extract damping in leading order estimates}  for the high-frequency modes using singularly weighted norms {in the energy estimates}.  The blowup analysis for the viscous model is more subtle since the viscous terms do not provide damping and produce some bad terms if we use a singularly weighted norm. We establish the blowup analysis for the viscous model by carefully designing an energy norm that combines a singularly weighted energy norm and a {sum of} high-order Sobolev norms.

\subsection{Problem setting}
In \cite{hou2008dynamic},
Hou-Li introduced the following reformulation of the axisymmetric Navier-Stokes equation:
\begin{eqnarray}
\label{3d}
&&u_{1, t}+u^{r} u_{1, r}+u^{z} u_{1, z} =2 u_{1} \psi_{1, z}+\nu\Delta u_1\,, \\
&&\omega_{1, t}+u^{r} \omega_{1, r}+u^{z} \omega_{1, z} =\left(u_{1}^{2}\right)_{z}+\nu\Delta \omega_1\,, \\
&&-\left[\partial_{r}^{2}+(3 / r) \partial_{r}+\partial_{z}^{2}\right] \psi_{1} =\omega_{1}\,,
\end{eqnarray}
where $u_1 = u^\theta/r,\; \omega_1 = \omega^\theta,\; \psi_1 = \psi^\theta/r, $ and $u^\theta$, $\omega^\theta$, and $\psi^\theta$ are the angular velocity, angular vorticity, and angular stream function, respectively. 
By the well-known Caffarelli-Kohn-Nirenberg partial regularity result \cite{caffarelli1982partial}, the axisymmetric Navier-Stokes equations can develop a finite time singularity only along the symmetry axis $r=0$. To study the potential singularity or global regularity of the axisymmetric Navier-Stokes equations,
Hou-Li \cite{hou2008dynamic} proposed the following 1D model along the symmetry axis $r=0$:
\begin{equation}
\label{1d}
\begin{aligned}
u_{1, t}+2 \psi_1 u_{1, z} &=2 \psi_{1, z} u_1 +\nu u_{1,zz}\,, \\
\omega_{1, t}+2 \psi_1 \omega_{1, z} &=\left(u_1^{2}\right)_{z}+\nu \omega_{1,zz}\,, \\
- \psi_{1,zz} &=\omega_1\,.
\end{aligned}
\end{equation}
Such a reduction is exact in the sense that if ($\omega_1$, $u_1$, $\psi_1$) is an exact solution of the 1D model, we can obtain an exact solution of the 3D Navier-Stokes equations by using a constant extension in $r$. This corresponds to the case when the physical quantities $u^{\theta}=ru_1$, $\omega^{\theta}=r\omega_1$ are linear in $r$.
We assume that the solutions are periodic in $z$ on $[0,2\pi]$.
We already know from the original Hou-Li paper that this system is well-posed for $C^m$ initial data with $m \geq 1$.  In \cite{hou2008dynamic}, the authors also used the well-posedness of the Hou-Li model to construct globally smooth solutions to the 3D equations with large dynamic growth.

In two recent papers by the first author \cite{hou2022potential, hou2022potentially}, the author presented new numerical evidence that the 3D axisymmetric Euler and Navier-Stokes equations develop potential singular solutions at the origin. This new blowup scenario is very different from the Hou-Luo blowup scenario, which occurs on the boundary. In this computation, the author observed that the axial velocity $u^{z}=2 \psi_{1}+r \psi_{1, r}$ near the maximal point of $u_1$ is significantly weaker than $2\psi_1$. This is due to the fact that $\psi_1$ reaches the maximum at a position $r=r_\psi$ that is smaller than the position $r=r_u$ in which  $u_1$ achieves its maximum, i.e. $r_\psi < r_u$. Therefore $\psi_{1, r}$ is negative near the maximal position of $u_1$. Thus the axial velocity $u^z$ is actually weaker than $2 \psi_1$, which corresponds to $u^z |_{r=0}$. Thus, the original Hou-Li model along $r=0$ does not capture this subtle phenomenon, which is three-dimensional in nature. To gain some understanding of this potentially singular behavior, we introduce the following 1D weak advection model.
\begin{equation}
\label{1dwp}
\begin{aligned}
u_{ t}+2 a\psi u_{ z} &=2 u \psi_{ z}+\nu u_{zz}\,, \\
\omega_{ t}+2 a\psi\omega_{ z} &=\left(u^{2}\right)_{z}+\nu \omega_{zz}\,, \\
- \psi_{zz} &=\omega\,,
\end{aligned}
\end{equation}
where $a$ is a parameter that measures the relative strength of advection in the Hou-Li model.
\begin{remark}
For simplicity, we drop the subscript $1$ in the above weak advection model.
The proposed model \eqref{1dwp} in the inviscid case $\nu=0$  resembles the generalized Constantin-Lax-Majda model (gCLM) \cite{okamoto2008generalization} 
$$\omega_t+au\omega_x=u_x\omega\,,\qquad u_x=H\omega\,,$$
        where
     $$
H \omega(x)=\frac{1}{\pi} p.v. \int_{\mathbb{R}} \frac{\omega(y)}{x-y} \mathrm{~d} y
$$
is the Hilbert transform. They share similar structures of competition between advection and vortex stretching. The case when $a=1$ corresponds to the De Gregorio (DG) model.
We obtain an explicit steady-state to the inviscid Hou-Li model \eqref{1d} $(\omega,u,\psi)=(\sin x,\sin x, \sin x)$, similar to the steady state $(\omega,u)=(-\sin x,\sin x)$ of the DG model on $S^1$. Many of the results we present in this paper have analogies for the gCLM model; see in particular \cite{chen2021regularity,chen2021slightly}.
\end{remark}

\subsection{Main results}
We summarize the main results of the paper below and devote the subsequent sections to proving these results.
Our first result is on the finite-time blowup of the weak inviscid advection model; for its proof see Section \ref{linearest}  and \ref{nonlinearest}. 
\begin{theorem}
\label{t1}
For the weak advection model \eqref{1dwp} in the inviscid case $\nu=0$, there exists a constant $\delta>0$ such that for $a\in(1-\delta,1)$, the weak advection model \eqref{1dwp} develops a finite time singularity for some $C^{\infty}$ initial data. Moreover, there exists a self-similar profile $({\omega}_{\infty},{\omega}_{\infty},{\omega}_{\infty})$ corresponding to a blowup that is neither expanding nor focusing. More precisely, the blowup solution to \eqref{1dwp} has the form
$$\omega(x,t)=\frac{1}{1+c_{u,\infty}t}{\omega}_{\infty}\,,u(x,t)=\frac{1}{1+c_{u,\infty}t}u_{\infty}\,,\psi(x,t)=\frac{1}{1+c_{u,\infty}t}{\psi}_{\infty}\,,$$
for some negative constant $c_{u,\infty}$ with a blowup time given by $T= \frac{-1}{c_{u,\infty}}$.
\end{theorem}
\begin{remark}
    Such self-similar blowup that is neither expanding nor focusing is observed numerically for $a\in[0.6,0.9]$. See also a similar phenomenon observed for the gCLM model in \cite{lushnikov2021collapse} for $a\in[0.68,0.95]$. The blowup result for the gCLM model has been proved in \cite{chen2021slightly} for $a$ sufficiently close to $1$. We remark that for $a$ very close to $1$, since \zm{we can show that $c_{u,\infty}=2(a-1)+o(a-1)$}, the blowup time becomes very large due to the very small coefficient {$1-a$} in the vortex stretching term {which slightly dominates the advection term}. It would be extremely difficult to compute such singularity numerically since it takes an extremely long time for the singularity to develop. For $a$ below a critical value $a_0$, i.e. $a < a_0$, we observe that the weak advection Hou-Li model develops a focusing singularity.
\end{remark}
The second result is on the blowup of the original Hou-Li model with $C^\alpha$ initial data; for its proof see Section \ref{sec hol}. In \cite{elgindi2020effects}, the authors made an important observation that advection can be weakened by $C^\alpha$ data. Intuitively if $u =O(x^\alpha)$ in the origin, since $\psi$ is $C^2$, we have that $\psi u_x\approx \alpha\psi_x u$ near the origin, the vortex stretching term is stronger than the advection term if $\alpha < 1$. See \cite{chen2021finite,chen2021regularity} on results of  blowup of the DG model with H\"older continuous data.
\zm{
\begin{theorem}
\label{t2}
Consider the Hou-Li model \eqref{1d} in the inviscid case $\nu=0$. There exists a constant $\delta_0>0$ such that for $\alpha\in(1-\delta_0,1)$, \eqref{1d} develops a finite time singularity for some $C^{\alpha}$ initial data. Moreover, there exists a $C^{\alpha}$ self-similar profile corresponding to a blowup that is neither expanding nor focusing, similar to the setting in Theorem \ref{t1}.
\end{theorem}
\begin{remark}
    This theorem establishes blowups of type $C^{\alpha}$ for any $\alpha$ close to 1, which of course implies blowups in less regular classes since $C^{\alpha}\subset C^{\alpha_1}$ for $\alpha_1<\alpha$. The regularity of the profile determines the speed of the blowup since our constructed $C^{\alpha}$ profile has blowup time $T=O(-1/(2(\alpha-1)))$. We remark that however, we do not have blowup for data intrinsically in a low regularity class $C^{\epsilon}$ for $\epsilon$ close to $0$; that is, data that is $C^{\epsilon}$ but not in any higher $C^\alpha$ classes. We conjecture that such blowup might be focusing, which is beyond the scope of this paper.
\end{remark}
\begin{remark}
    The above two theorems imply that the result of the wellposedness in \cite{hou2008dynamic} of the Hou-Li model for $C^1$ initial data is sharp. As long as the advection is weakened or slightly less smooth data is allowed, we would have a self-similar blowup.
\end{remark}}
The third result is on the finite-time blowup of the weak advection model with viscosity. The dynamic rescaling formulation implies that the viscous terms are asymptotically small. Thus, we can build on Theorem \ref{t1} to establish Theorem \ref{t3}. \zm{We remark that there is no exact self-similar profile due to the viscous term.}
We will provide more details of the blowup analysis for the viscous case in Section \ref{sec5}.
\begin{theorem}
\label{t3}
Consider the weak advection model \eqref{1dwp} with viscosity. There exists a constant $\delta_1>0$ such that for $a\in(1-\delta_1,1)$, the weak advection model \eqref{1dwp} develops a finite time singularity for some $C^{\infty}$ initial data.
\end{theorem}

We use the framework of the dynamic rescaling formulation to establish the blowups. This formulation was first introduced by McLaughlin, Papanicolaou, and co-workers in their study of self-similar blowup of the nonlinear Schr\"odinger equation  \cite{mclaughlin1986focusing,landman1988rate}.  This formulation was later developed into an effective modulation technique, which has been applied to analyze the singularity formation for the nonlinear Schr\"odinger equation \cite{kenig2006global,merle2005blow}, compressible Euler equations \cite{buckmaster2019formation},  the nonlinear heat equation \cite{merle1997stability}, the generalized KdV equation \cite{martel2014blow}, and other dispersive problems. Recently this approach has been applied to prove singularity in various gCLM models \cite{chen2021finite,chen2021regularity,chen2021slightly} and in Euler equations \cite{ elgindi2021finite, chen2019finite2, chen2022stable}. Our blowup analysis consists of several steps.  First, we use the dynamic rescaling formulation to link a self-similar singularity to the (stable) steady state of the dynamic rescaling formulation. Secondly, we identify, either analytically or numerically, an approximate steady state to the dynamic rescaling formulation. Thirdly, we perform energy estimates using a singularly weighted norm to establish linear and nonlinear stability of the approximate steady state. Finally, we establish exponential convergence to the steady state in the rescaled time. 

The crucial ingredient of the framework is the linear stability of the approximate steady state, and we usually adopt a singularly weighted $L^2$-based estimate. {To avoid an overestimate in the linear stability analysis, we expand the perturbation in terms of the orthonormal basis with respect to the weight $L^2$ norm and reduce the linear stability estimate into an estimate of a quadratic form for the Fourier coefficients. We further extract the damping effect of the linearized operator by establishing a lower bound on the eigenvalues of an infinite-dimensional symmetric matrix. We prove the positive-definiteness of this quadratic form by performing an exact computation of the eigenvalues of a small number of Fourier modes with rigorous computer-assisted bounds}, and treat the high-frequency Fourier modes as a small perturbation by using the asymptotic decay of {the quadratic form in} the high-frequency Fourier coefficients.

\subsection{Organization of the paper and notations}
In Section \ref{linearest}, we introduce our dynamic rescaling formulation and link the blowup of the physical equation to the steady state of the dynamic rescaling formulation. The linear stability of the approximate steady state is established.
In Section \ref{nonlinearest}, we establish the nonlinear stability of the approximate steady state and the exponential convergence to the steady state, which proves Theorem \ref{t1} and the blowup for the weak advection model. In Section \ref{sec hol}, we prove Theorem \ref{t2} and establish blowup for the original model with H\"older continuous data. In Section \ref{sec5}, we prove Theorem \ref{t3} by designing a special energy norm to estimate the viscous terms. We provide the crucial linear damping estimates in  the Appendix using computer assistance.

Throughout the article, we use $(\cdot,\cdot)$ to denote the inner product on $S^1$: $(f,g)=\int_{-\pi}^{\pi}fg$. We use $C$ to denote absolute constants, which may vary from line to line, and we use $C(k)$  to denote some constant that may depend on specific parameters $k$  we choose. We use $A\lesssim B$ for positive $B$ to denote that there exists an absolute constant $C>0$ such that $A\leq CB$.
\section{Dynamic rescaling formulation and linear estimates}
\label{linearest}
\subsection{Dynamic rescaling formulation}
\label{drf sec}
We will establish the singularity formation of the weak advection model by using the dynamic rescaling formulation. We first consider the inviscid case with $\nu=0$. For solutions to the system \eqref{1dwp}, we introduce $$\tilde{u}(x, \tau)=C_{u}(\tau)  u( x, t(\tau))\,,\quad \tilde{\omega}(x, \tau)=C_{u}(\tau)  \omega( x, t(\tau))\,,\quad \tilde{\psi}(x, \tau)=C_{u}(\tau)  \psi(x, t(\tau))\,,$$ where $$
C_{u}(\tau)=\exp \left(\int_{0}^{\tau} c_{u}(s) d s\right)\,, \quad t(\tau)=\int_{0}^{\tau} C_{u}(s) d s\,.$$
We can show that the rescaled variables solve the following dynamic rescaling equation \begin{equation}
\label{1drf}
\begin{aligned}
\tilde{u}_{\tau}+2a \tilde{\psi}\tilde{u}_{ x} &=2 \tilde{u} \tilde{\psi}_{ x}+c_u \tilde{u}\,, \\
\tilde{\omega}_{ \tau}+2a \tilde{\psi}\tilde{\omega}_{ x} &=\left(\tilde{u}^{2}\right)_{x}+c_{u}\tilde{\omega}\,, \\
-\tilde{\psi}_{xx} &=\tilde{\omega}\,.
\end{aligned}
\end{equation}

\begin{remark}
We do not rescale the spatial variable $x$, since we are interested in a blowup solution that is neither focusing nor expanding within a fixed period. The scaling factors for $u$, $\omega$, $\psi$ are thus the same.
\end{remark}
When we establish a self-similar blowup, it suffices to show the dynamic stability of equation \eqref{1drf} close to an approximate steady state  with scaling parameter $c_u<-\epsilon<0$ uniformly in time for a small constant $\epsilon$; see also \cite{chen2021finite}. In fact, it's easy to see that if $(\tilde{u},\tilde{\omega},\tilde{\psi},c_u)$ converges to a steady-state $({u}_{\infty},{\omega}_{\infty},{\psi}_{\infty},c_{u,\infty})$ of \eqref{1drf}, then $$\omega(x,t)=\frac{1}{1+c_{u,\infty}t}{\omega}_{\infty}\,,u(x,t)=\frac{1}{1+c_{u,\infty}t}u_{\infty}\,,\psi(x,t)=\frac{1}{1+c_{u,\infty}t}{\psi}_{\infty}\,,$$
is a self-similar solution of \eqref{1dwp}. 
    
From now on, we will primarily work in the dynamic rescaling formulation and use the notations that $\tilde{u} =\bar{u}+\hat{u}$, where $\bar{u}$ is the approximate steady state that we perturb around and $\hat{u}$ is the perturbation. Notations for variables $\tilde{\omega}$ and $\tilde{\psi}$ are similar. 
\subsection{Equations governing the perturbation}
We use the steady state corresponding to the case of $a=1$ to construct an approximate steady state for \eqref{1drf}.
$$\bar{\omega}=\sin x\,,\quad\bar{u}=\sin x\,,\quad\bar{\psi}=\sin x\,,\quad \bar{c}_u=2(a-1)\bar{\psi}_x(0)=2(a-1)\,.$$
We consider odd perturbations $\hat{u}$, $\hat{\omega}$, $\hat{\psi}$. The parities are preserved in time by equation \eqref{1drf}.  We use the normalization condition as $c_u=2(a-1)\hat{\psi}_x(0)$. This normalization ensures that $\bar{u}_x(0)+\hat{u}_x(0)$ is conserved in time. 

To simplify our presentation, we will drop the $\hat{}$ in the perturbation $\hat{u}$ and use $u$ for $\hat{u}$, $\omega$ for $\hat{\omega}$, $\psi$ for $\hat{\psi}$.   Now the perturbations satisfy the following system
\begin{equation}
\label{1drfr}
\begin{aligned}
{u}_{ \tau} &=-2a \sin x u_x-2a \cos x \psi+ 2 u \cos x+2\sin x \psi_x+\bar{c}_u u+ c_u\bar{u}+N_1+F_1\,, \\
{\omega}_{ \tau}&=-2a \sin x \omega_x-2a \cos x\psi +2 u \cos x+2\sin x u_x+\bar{c}_u \omega+ c_u\bar{\omega}+N_2+F_2\,, \\
-{\psi}_{xx} &={\omega}\,,
\end{aligned}
\end{equation}
where $N_1, \; N_2$ and $F_1, \; F_2$ are the nonlinear terms and error terms defined below:
$$N_1=(c_u+2\psi_x)u-2a\psi u_x\,,\quad N_2=c_u\omega+2uu_x-2a\psi \omega_x\,,$$
$$F_1=(\bar{c}_u+2\bar{\psi}_x)\bar{u}-2a\bar{\psi} \bar{u}_x=2(a-1)\sin x(1-\cos x)\,,\quad F_2=\bar{c}_u\bar{\omega}+2\bar{u}\bar{u}_x-2a\bar{\psi} \bar{\omega}_x=F_1\,.$$
We further organize the system \eqref{1drfr} into the main linearized term and a smaller term containing a factor of $a-1$:
\begin{equation}
\label{1drc}
\begin{aligned}
{u}_{ \tau} &=L_1+(a-1)L'_1+N_1+F_1\,, \\
{\omega}_{ \tau}&=L_2+(a-1)L'_2+N_2+F_2\,, \\
-{\psi}_{xx} &={\omega}\,.
\end{aligned}
\end{equation}
where
$$L_1=-2 \sin x u_x-2 \cos x \psi+ 2 u \cos x+2\sin x \psi_x\,,$$ $$ L'_1=-2\sin x u_x-2 \cos x \psi+2u+2\psi_x(0)\sin x\,,$$ $$L_2=-2 \sin x \omega_x-2 \cos x\psi +2 u \cos{x}+2\sin x u_x\,,$$ $$ L'_2=-2 \sin x \omega_x-2 \cos x\psi+2\omega+2\psi_x(0)\sin x \,.$$

To show that the dynamic rescaling equation is stable and converges to a steady state, we will perform a weighted-$L^2$ estimate with a singular weight $\rho$ and a weighted $L^2$ norm$$\rho=\frac{1}{2\pi(1-\cos x)}\,,\quad \|f\|_{\rho}=(f^2,\rho)^{1/2}\,.$$ For initial perturbation with $u_x(0,0)=0$, we have $u_x(0,\tau)=0$ for all time and 
$$E^2(\tau)=\frac{1}{2}((u_x^2,\rho)+(\omega^2,\rho))\; ,$$ 
is well-defined. We will first show that the dominant parts $L_1$ and $L_2$ provide damping.  The following lemma is crucial and motivates the choice of $\rho$.
\begin{lemma}
\label{dampl}We have the following identity
$$(\sin x f_x,f\rho)=\frac{1}{2}(f^2,\rho)\,,$$
which can be verified directly by using integration by parts.
\end{lemma}
\subsection{Stability of the main parts in the linearized equation}
 \label{stab main}

In order to extract the maximal amount of damping, we will expand the perturbed solution in the Fourier series and perform exact calculations. We first explore the orthonormal basis in $L^2(\rho)$.
\begin{lemma}
For the space of odd periodic functions on $[0,2\pi]$, we describe a complete set of orthonormal basis $\{o^k\}$ in $L^2(\rho)$  $$o^k=\sin(kx)-\sin((k-1)x)\,,\quad k=1,2,\cdots\,.$$
Similarly, for the space of even periodic functions that lie in $L^2(\rho)$, we describe a complete set of orthonormal basis $\{e^k\}$
$$e^k=\cos(kx)-\cos((k+1)x)\,,\quad k=0,1,\cdots\,.$$
\end{lemma}
Now we are now ready to establish linear stability. 
\begin{proposition}\label{prop d} The following energy estimate holds for the leading linearized operators
$$dE_1\coloneqq((L_1)_x,u_x\rho)+(L_2,\omega\rho)\leq -0.16[(u_x,u_x\rho)+(\omega,\omega\rho)]\,.$$
\end{proposition}
\begin{proof}
Consider the expansion of $\omega$, $u_x$, and $u$ in the orthonormal basis
$$\omega=\sum_{k\geq 1} a_k o^k\,, \quad u=\sum_{k\geq 1} b_k o^k\,, \quad u_x=\sum_{k\geq1} c_k e^k\,.$$
Note the summation index for $u_x$ satisfies $k \geq 1$ since we can easily see that $$ (u_x,e^0\rho)=\frac{1}{2\pi}\int_0^{2\pi} u_x=0\,.$$We first express $b_k$ in terms of $c_k$. If we insert the expression of the basis into $u$, take derivative and compare the coefficients with the expansion of $u_x$, we get $$c_i=\sum_{k=1}^i b_k-ib_{i+1}\,.$$
Therefore we can solve $$b_{i+1}=b_1-\sum_{k=1}^{i-1}\frac{c_k}{k(k+1)}-\frac{c_{i}}{i}\,.$$
Moreover, we have the compatibility condition $u_x(0)=\sum_{k \geq 0} b_k=0$. Therefore we can solve $b_1$ and obtain 
\begin{equation}\label{bc}
    b_{i}=\sum_{k\geq i}\frac{c_k}{k(k+1)}-\frac{c_{i-1}}{i}\,,
\end{equation}
where we define $c_0=0$.

Now we write out the terms explicitly using the expansions
\begin{equation*}
\begin{aligned}
dE_1&=2(-u\sin x-u_{xx}\sin x,u_x\rho)+2(\sin x\psi+\sin x\psi_{xx},u_x\rho)\\&+2(- \sin x \omega_x- \cos x\psi,\omega \rho) +2 (u \cos{x}+\sin x u_x,\omega \rho)\\&=-[(u_x,u_x\rho)+(\omega,\omega\rho)+(u,u\rho)]+2[-( \cos x\psi,\omega \rho)+(\sin x\psi,u_x\rho)+ (u \cos{x},\omega \rho)]\,.    
\end{aligned}    
\end{equation*}

Here we use the crucial Lemma \ref{dampl} to extract damping on the local terms and the Biot-Savart law $-\psi_{xx}=\omega$ to cancel the effect of the nonlocal terms {$\sin x\psi_{xx}$ in $(L_1)_x$ and $\sin xu_{x}$ in $L_2$}. Next, we calculate the remaining nonlocal terms explicitly.
$$-2( \cos x\psi,\omega \rho)=(2(1-\cos x)\psi,\omega \rho)-2(\psi,\omega \rho)=\frac{1}{\pi}(\psi,\omega)-2(\psi,\omega \rho)\,.$$
We express $\omega$ and $\psi$ both in terms of orthonormal basis $o^k$ corresponding to the weighted norm and the canonical basis $\sin(kx)$ corresponding to the (normalized by $\frac{1}{\pi}$) $L^2$ norm.
$$\omega=\sum_{k \geq 1} (a_k-a_{k+1}) \sin(kx)\,,$$
where we denote $a_0=0$. Therefore 
$$\psi=\sum_{k \geq 1} \frac{a_k-a_{k+1}}{k^2} \sin(kx)\,.$$
Furthermore, we collect
$$\psi=\sum_{k\geq 1} \sum_{j\geq k}\frac{a_j-a_{j+1}}{j^2} o^k\,.$$
Therefore we can compute explicitly that 
$$-2( \cos x\psi,\omega \rho)=\sum_{k \geq 1} \frac{(a_k-a_{k+1})^2}{k^2}-2\sum_{k\geq 1} a_k\sum_{j\geq k}\frac{a_j-a_{j+1}}{j^2}\,.$$

We use integration by parts similar to Lemma \ref{dampl} to obtain 
$$2[(\sin x\psi,u_x\rho)+ (u \cos{x},\omega \rho)]=2(\cos{x}\omega+\psi-\sin x\psi_x,u\rho)\coloneqq2(T_u,u\rho)\,.$$
We further have $$\begin{aligned}
T_u&=\sum_{k \geq 1} (a_k-a_{k+1})[\frac{k+1}{2k}\sin((k-1)x)+\frac{k-1}{2k}\sin((k+1)x)+\frac{1}{k^2}\sin(kx)]\\&=\sum_{k \geq 1} \sin(kx) [\frac{k+2}{2(k+1)}(a_{k+1}-a_{k+2})+\frac{a_{k}-a_{k+1}}{k^2}+\frac{k-2}{2(k-1)}(a_{k-1}-a_{k})]\\&=\sum_{k\geq 1}[\frac{k-2}{2(k-1)}a_{k-1}+(\frac{1}{2k(k-1)}+\frac{1}{k^2})a_{k}+(\frac{k+1}{2k}+\frac{1}{(k+1)^2}-\frac{1}{k^2})a_{k+1}\\&+\sum_{j>k+1}(\frac{1}{j^2}-\frac{1}{(j-1)^2})a_j]o^k\,,  
\end{aligned}$$
where the terms involving $\frac{1}{k-1}$ in the summand is regarded as $0$ for $k=1$. Therefore we collect explicitly that 
$$\begin{aligned}
dE_1&=\sum_{k\geq 1}\{ -(a_k^2+b_k^2+c_k^2)+ \frac{(a_k-a_{k+1})^2}{k^2}-2 a_k\sum_{j\geq k}\frac{a_j-a_{j+1}}{j^2}+2b_k[\frac{k-2}{2(k-1)}a_{k-1}\\&+(\frac{1}{2k(k-1)}+\frac{1}{k^2})a_{k}+(\frac{k+1}{2k}+\frac{1}{(k+1)^2}-\frac{1}{k^2})a_{k+1}+\sum_{j>k+1}(\frac{1}{j^2}-\frac{1}{(j-1)^2})a_j]\}\,.\end{aligned}$$
Substituting \eqref{bc} into the above and we can simplify
$$
\begin{aligned}
dE_1&=-\sum_{k \geq 1}\{ a_k^2(1+\frac{1}{k^2}-\frac{1}{(k-1)^2})+c_k^2(1+\frac{1}{k(k+1)})+ 2a_k a_{k+1}\frac{1}{(k+1)^2}\\&+2 a_k \sum_{j>k+1}a_j(\frac{1}{j^2}-\frac{1}{(j-1)^2})+2a_kc_k\frac{1+2k-k^2}{2k^2(k+1)}+2a_{k+1}c_k\frac{k^2-k-1}{2k^2(k+1)^2}\\&-2a_{k+2}c_k\frac{k+2}{2(k+1)^2}+\sum_{j>k}2a_k c_j\frac{1}{j(j+1)}\}\,.\end{aligned}$$
After this explicit computation, we notice that the damping estimate in Proposition \ref{prop d} can be cast into an estimate of a quadratic form; see \eqref{dample}, which is equivalent to a lower bound on the eigenvalues of an infinite-dimensional symmetric matrix. 
\begin{equation}
    \label{dample}\begin{aligned}
 F(a,c)&\coloneqq \sum_{k\geq1}\{ a_k^2(0.84+\frac{1}{k^2}-\frac{1}{(k-1)^2})+c_k^2(0.84+\frac{1}{k(k+1)})+ 2a_k a_{k+1}\frac{1}{(k+1)^2}\\&+2 a_k \sum_{j>k+1}a_j(\frac{1}{j^2}-\frac{1}{(j-1)^2})+2a_kc_k\frac{1+2k-k^2}{2k^2(k+1)}+2a_{k+1}c_k\frac{k^2-k-1}{2k^2(k+1)^2}\\&-2a_{k+2}c_k\frac{k+2}{2(k+1)^2}+\sum_{j>k}2a_k c_j\frac{1}{j(j+1)}\}\geq0\,.\end{aligned}\end{equation}

We notice that the entries decay fast. {Therefore the strategy to prove \eqref{dample} is to combine a computer-assisted estimate of the eigenvalues of its finite truncation with a decay estimate of the remaining part. We will defer the proof of \eqref{dample} to the Appendix, see Lemma \ref{cap} and the proof}. Thereby we conclude the linear estimate.
\end{proof}

\section{Nonlinear estimates and convergence to self-similar profile}
\label{nonlinearest}
\subsection{Nonlinear stability}
By Proposition \ref{prop d} and equation \eqref{1drc}, we have \begin{equation}
\label{nonlinear}
    \begin{aligned}\frac{1}{2}\frac{d}{d\tau}E^2(\tau)&\leq -0.16E^2(\tau)+(a-1)[((L'_1)_x,u_x\rho)+(L'_2,\omega\rho)]\\&+((N_1)_x,u_x\rho)+(N_2,\omega\rho)+((F_1)_x,u_x\rho)+(F_2,\omega\rho)\,.\end{aligned}
\end{equation}
We first provide some estimates about the weighted $L^2$ norm and $L^{\infty}$ norm of some lower-order terms.
\begin{lemma}
\label{lotb}
The following estimates hold
\begin{enumerate}
    \item Weighted $L^2$ norm: $$\|\psi\|_{\rho},\|\psi_x-\psi_x(0)\|_{\rho},\|u\|_{\rho}\lesssim E\,.$$
    \item $L^{\infty}$ norm: $$\|\psi_x\|_{\infty},\|\frac{\psi}{\sin x}\|_{\infty},\|u\|_{\infty}\lesssim E\,.$$
\end{enumerate}
   
\end{lemma}
\begin{proof}
    For (1), we use the setting of the Fourier series approach as in the proof of Proposition \ref{prop d} and pick up the notation there.
    $$\begin{aligned}
    \|\psi_x-\psi_x(0)\|_{\rho}^2&=\sum_{k \geq 1} (\sum_{j\geq k}\frac{a_j-a_{j+1}}{j})^2\leq\sum_{k \geq 1} \sum_{j\geq k}a^2_j(\frac{1}{k^2}+\sum_{j>k}(\frac{1}{j}-\frac{1}{j+1})^2)\\&\lesssim\sum_{j \geq 1} a_j^2\sum_{k \geq 1} \frac{1}{k^2}\lesssim\|\omega\|_{\rho}^2\,,\end{aligned}$$
    where we have used the Cauchy-Schwarz inequality. We can similarly estimate $\|\psi\|_{\rho}$. Then we get
    $$\|u\|_{\rho}^2=\sum_{j \geq 1} b_j^2=\sum_{j \geq 1} \frac{1}{j(j+1)}c_j^2\leq\sum_{j \geq 1} c_j^2=\|u_x\|_{\rho}^2\,.$$
    For (2), we first compute using Fourier series similar to (1) $$\psi_x(0)=\sum_{j \geq 1}\frac{a_j-a_{j+1}}{j}\lesssim\|\omega\|_{\rho}\,.$$
    Next, we estimate $$\|\psi_x-\psi_x(0)\|_{\infty}\lesssim\|\psi_{xx}\|_{1}\lesssim\|\omega\|_{\rho}\,.$$
Similarly, we obtain the estimate for $\|u\|_{\infty}$.
    For $\|\frac{\psi}{\sin x}\|_{\infty}$, since $\psi$ is odd and periodic, we have $\psi(\pi)=\psi(0)=0$ and only need to estimate this norm in $[0,\pi]$. Since $\sin x\geq \frac{2}{\pi}\min\{x,\pi-x\}$ in $[0,\pi]$,
    we have $\|\frac{\psi}{\sin x}\|_{\infty}\lesssim\|\psi_x\|_{\infty}$ by Lagrange's mean value theorem.
\end{proof}
Combined with the damping in Lemma \ref{dampl}, we further obtain 
$$\begin{aligned}((L'_1)_x,u_x\rho)&=2(-\cos xu_x-\sin xu_{xx}+\sin x\psi-\cos x\psi_x+u_x+\psi_x(0)\cos x,u_x\rho)\\&\lesssim E^2+(\|\psi\|_{\rho}+\|\psi_x-\psi_x(0)\|_{\rho})E\lesssim E^2\,,\end{aligned}$$
$$(L'_2,\omega\rho)=2(-\sin x\omega_x-\cos x\psi+\omega+\psi_x(0)\sin x,\omega\rho)\lesssim E^2+(\|\psi\|_{\rho}+|\psi_x(0)|)E\lesssim E^2\,.$$
$$((F_1)_x,u_x\rho)\lesssim |a-1|E\,,\quad(F_2,\omega\rho)\lesssim |a-1|E\,,$$
$$\begin{aligned}((N_1)_x,u_x\rho)&=2(-\omega u+(1-a)(\psi_x-\psi_x(0))u_{x}-a\psi u_{xx},u_x\rho)\\&\lesssim E^2(\|\psi_x\|_{\infty}+\|u\|_{\infty})+|(u_x^2,(\psi\rho)_x)|\\&\lesssim E^2(\|\psi_x\|_{\infty}+\|u\|_{\infty}+\|\frac{\psi}{\sin x}\|_{\infty})\lesssim E^3\,,\end{aligned}$$
$$(N_2,\omega\rho)=2((a-1)\psi_x(0)\omega+uu_x-a\psi\omega_x,\omega\rho)\lesssim E^2(\|\psi_x\|_{\infty}+\|u\|_{\infty}+\|\frac{\psi}{\sin x}\|_{\infty})\lesssim E^3\,.$$
Therefore we have \begin{equation}
\label{nonl-collected}
    \frac{d}{d\tau}E(\tau)\leq -(0.16-C|a-1|)E+C|a-1|+CE^2\,.
\end{equation}
We can perform the standard bootstrap argument to show that there exist absolute constants $\delta,C>0$ such that if $|a-1|<\delta$ and $E(0)<C|a-1|$, then we have $E(\tau)<C|a-1|$ for all time. In particular $c_u=O(|a-1|^2)$ and $c_u+\bar{c}_u<0$. Therefore we prove that the solution blows up in finite time.
\subsection{Estimates using a higher-order Sobolev norm}
\label{h2sec}
In order to establish convergence of the solution to a steady state, we need to estimate weighted norms of $u_t$ and $\omega_t$. As was pointed out in \cite{chen2021slightly}, we need to provide stability estimates of the equation in higher-order Sobolev norms to close the estimate. In particular, we choose $$K^2(\tau)=\|D_xu_x\|_{\rho}^2+\|D_x\omega\|_{\rho}^2\,,$$
where we denote $D_x$ to be the operator $\sin x\partial x$.

\begin{remark}
This choice of weighted  norms is again motivated by the local linear damping estimates. We recall that the leading order terms of the local terms in the linearized operators $(L_1)_x$, $L_2$ are $-2D_xu_x$ and $-2D_xw$, and we have $2(D_x f,f\rho)=(f,f\rho)$. Therefore in this new weighted norm, the combined terms would again give damping \begin{equation}\label{highd}
    (-2D_x D_x u_x,D_x u_x\rho)+(-2D_xD_x w,D_x w\rho)=-K^2\,.
\end{equation}
\end{remark}
We now obtain
$$\begin{aligned}\frac{1}{2}\frac{d}{d\tau}K^2(\tau)&\leq (D_x(L_1)_x,D_xu_x\rho)+(D_xL_2,D_x\omega\rho)+(D_x(N_1)_x,D_xu_x\rho)\\&+(D_xN_2,D_x\omega\rho)+(a-1)[(D_x(L'_1)_x,D_xu_x\rho)+(D_xL'_2,D_x\omega\rho)]\\&+(D_x(F_1)_x,D_xu_x\rho)+(D_xF_2,D_x\omega\rho)\,.\end{aligned}$$
We will denote the terms that have $\|\cdot\|_{\rho}$ norm bounded by $E$ as $\textit{l.o.t.}$. The bound
$$\|D_x[fg]\|_{\rho}\lesssim (\|f_x\|_{2}+\|f\|_{2})\quad \text{for}\, g=1,\cos x,\sin x$$
combined with the oddness of $\psi$ and $u$ would imply that $D_x[fg]$ is $\textit{l.o.t.}$ for $f=\psi,\psi_x,u$ and $g=\sin x,\cos x,1$.
Therefore combined with \eqref{highd}, we have the following estimate for the main term $$dK_1\coloneqq(D_x(L_1)_x,D_xu_x\rho)+(D_xL_2,D_x\omega\rho)\,,$$ $$\begin{aligned}dK_1&\leq -K^2-2(D_x[\sin x \omega],D_xu_x\rho)+2(D_xD_xu,D_x\omega\rho)+CEK\\&=-K^2-(\sin 2x\omega,D_xu_x\rho)+(\sin 2x u_x,D_x\omega\rho)+CEK\\&\leq-K^2+CEK\,,\end{aligned}$$
where we have again used a crucial cancellation in the equality, similar to that of $dE_1$ in Subsection \ref{stab main}. We estimate the rest of the terms similar to the nonlinear stability estimates in \eqref{nonlinear}.
$$(D_x(L'_1)_x,D_xu_x\rho)+(D_xL'_2,D_x\omega\rho)\lesssim K^2+EK\,,$$
$$(D_x(F_1)_x,D_xu_x\rho)+(D_xF_2,D_x\omega\rho)\lesssim|a-1|K\,,$$
$$\begin{aligned}(D_x(N_1)_x,D_xu_x\rho)&\lesssim EK^2+|(-2wD_xu+2(1-a)D_x\psi_x u_x-2a\psi D_xu_{xx},D_xu_x\rho)|\\&\lesssim EK^2+\|\sin xu_x\|_{\infty}EK+|(u^2_{xx},(\psi\sin^2x\rho)_x)|\\&\lesssim EK(K+E)+K^2\|\frac{\psi}{\sin x}\|_{\infty}\lesssim EK(K+E)\,,\end{aligned}$$
$$(D_xN_2,D_x\omega\rho)\lesssim EK^2+|(2D_xu u_x-2a\psi D_x\omega_{x},D_x\omega\rho)|\lesssim EK(K+E)\,,$$
where we have used integration by parts and the estimate $$\|\sin xu_x\|_{\infty}\lesssim\|\sin xu_{xx}\|_{1}+\|\cos xu_{x}\|_{1}\lesssim\|D_xu_{x}\|_{\rho}+\|u_{x}\|_{\rho}\lesssim E+K\,.$$

We can finally prove that $$\frac{d}{d\tau}K(\tau)\leq -(1-C|a-1|)K+CE+C|a-1|+CE(E+K)\,.$$
Therefore combined with \eqref{nonl-collected}, we can find an absolute constant $\mu>1$ such that $$\frac{d}{d\tau}(K+\mu E)\leq-(0.1-C|a-1|)(K+\mu E)+C|a-1|+C(K+\mu E)^2\,.$$
By using a standard bootstrap argument, there exist absolute constants $\delta_0<\delta,C>0$, if $|a-1|<\delta_0$ and $K(0)+\mu E(0)<C|a-1|$, then $K(\tau)+\mu E(\tau)<C|a-1|$ for all time.
\subsection{Convergence to the steady state}
\label{sec ct}
We estimate the weighted norm of $\omega_\tau$ and $u_{\tau,x}$ and then use the standard convergence in time argument as in \cite{chen2021finite,chen2021slightly}.
$$J^2(\tau)=\frac{1}{2}((u_{\tau,x}^2,\rho)+(\omega_{\tau}^2,\rho))\,.$$
Applying the estimates of $\frac{d}{d\tau}E$ to $\frac{d}{d\tau}J$, we can get damping for the linear parts, and the small error terms corresponding to $\bar{\omega}$ and $\bar{u}$ vanishes. Therefore we yield
$$\frac{1}{2}\frac{d}{d\tau}J^2\leq-(0.16-C|a-1|)J^2+((N_1)_{\tau,x},u_{\tau,x}\rho)+((N_2)_\tau,\omega_t\rho)\,.$$ 
 Using estimates similar to Lemma \ref{lotb} and nonlinear estimates in \eqref{nonlinear}, we get
$$((N_1)_{\tau,x},u_{\tau,x}\rho)\lesssim EJ^2+(\psi_\tau u_{xx},u_{\tau,x}\rho)\lesssim EJ^2+J^2\|u_{xx}\sin x\|_{\rho}\leq(E+K)J^2\,.$$
$$((N_2)_{\tau},\omega_\tau\rho)\lesssim EJ^2+(\psi_\tau \omega_{x},\omega_{\tau}\rho)\lesssim EJ^2+J^2\|\omega_{x}\sin x\|_{\rho}\leq(E+K)J^2\,.$$
Combined with the a priori estimates on $E+K$, we can establish exponential convergence of $J$ to zero. Then we can use the same argument as in \cite{chen2021finite,chen2021slightly} to establish exponential convergence to the steady state and conclude the proof of Theorem \ref{t1}.
\section{Blowup of the original model with H\"older continuous data}
\label{sec hol}
In this section, we follow the strategy of the linear and nonlinear estimates of the weak advection model in Sections \ref{linearest} and \ref{nonlinearest}, and establish blowup of $C^{\alpha}$ data for the original model \eqref{1dwp} with $a=1$. Here $\alpha<1$ is close to $1$. Many of the ideas are drawn from the paper \cite{chen2021regularity} and we only outline the most important steps. Intuitively, $C^{\alpha}$ regularity of the profile weakens the advection and therefore contributes to a blowup in finite time.
\subsection{Dynamic rescaling formulation around the approximate steady state}
Before we start, we will solve the Biot-Savart law of recovering $\psi$ from $\omega$ with odd symmetry.
\begin{lemma}
\label{alpha est}
    Suppose that $\omega$, $\psi$ are odd and periodic on $[-\pi,\pi]$, with $-\psi_{xx}=\omega$. Then we solve $\psi_x(0)=-\frac{1}{2\pi}\int_{0}^{2\pi}y{\omega}(y)$ and obtain \begin{equation}\label{bs-lawsolved}
        \psi=\int_{0}^{x}(y-x){\omega}(y)dy+x\psi_x(0)\,.
    \end{equation}
\end{lemma}
The proof of this lemma is straightforward by integration in $x$.

We construct the following approximate steady state with $C^{\alpha}$ regularity for \eqref{1drf}.
$$\bar{\omega}_{\alpha}=sgn(x)|\sin x|^{\alpha}\,,\quad\bar{u}_{\alpha}=sgn(x)|\sin x|^{\frac{1+\alpha}{2}}\,,\quad \bar{c}_{u, {\alpha}}=(\alpha-1)\bar{\psi}_{{\alpha},x}(0)\,,$$
where $\bar{\psi}_{\alpha}$ is related to $\bar{\omega}_{\alpha}$ via \eqref{bs-lawsolved}.
We consider odd perturbations $u$, $\omega$, $\psi$. The odd symmetry of the solution is preserved in time by equation \eqref{1drf}.  We will use the normalization condition as $c_u=(\alpha-1){\psi}_x(0)$, which ensures that $u$ vanishes to a higher order at all times so that we can use the same singular weight $\rho$. In fact we compute using \eqref{1drf} and the normalization conditions that $$\lim_{x\to 0}\frac{\bar{u}_{{\alpha}}(x)+u(x,\tau)}{x^{\frac{1+\alpha}{2}}}=\lim_{x\to 0}\frac{\bar{u}_{{\alpha}}(x)+u(x,0)}{x^{\frac{1+\alpha}{2}}}\,.$$ 
Therefore if we make the initial perturbation $u(x,0)$ vanish to order $1+\alpha$ around the origin,  $u(x,\tau)$ will also vanish to order $1+\alpha$ for all time. 

Now similar to what we obtain in \eqref{1drfr}, the perturbations satisfy the following system
\begin{equation}
\label{1drfr-alpha}
\begin{aligned}
{u}_{ \tau} &=L_1+R_{1,\alpha}+N_{1,\alpha}+F_{1,\alpha}\,, \\
{\omega}_{ \tau}&=L_2+R_{2,\alpha}+N_{2,\alpha}+F_{2,\alpha}\,, \\
-{\psi}_{xx} &={\omega}\,,
\end{aligned}
\end{equation}
where we extract the same leading order linear parts  as in \eqref{1drc}, while the nonlinear and error terms change and the residual error terms $R_{1,\alpha}$, $R_{2,\alpha}$ model the discrepancy between our approximate profile with $C^{\alpha}$ regularity and the steady state profile $\bar{\omega}=\bar{u}=\bar{\psi}=\sin x$. Define $\psi_{\text{res}}=\bar{\psi}_{\alpha}-\bar{\psi},\omega_{\text{res}}=\bar{\omega}_{\alpha}-\bar{\omega},u_{\text{res}}=\bar{u}_{\alpha}-\bar{u}$. We can express $R_{i,\alpha}$ and $F_{i,\alpha}$ as follows
$$R_{1,\alpha}=-2 \psi_{\text{res}} u_x-2 u_{\text{res},x} \psi+ 2 u \psi_{\text{res},x}+2u_{\text{res}} \psi_x+\bar{c}_{u,\alpha} u+ c_u\bar{u}_\alpha
\,.\,$$ 
$$R_{2,\alpha}=-2\psi_{\text{res}} \omega_x-2\omega_{\text{res},x}\psi +2 u u_{\text{res},x}+2u_{\text{res}} u_x+\bar{c}_{u,\alpha} \omega+c_u\bar{\omega}_{\alpha}\,,$$
$$N_{1,\alpha}=(c_u+2\psi_x)u-2\psi u_x\,,\quad N_{2,\alpha}=c_u\omega+2uu_x-2\psi \omega_x\,,$$
$$F_{1,\alpha}=(\bar{c}_{u,\alpha}+2\bar{\psi}_{\alpha,x})\bar{u}_{\alpha}-2\bar{\psi}_{\alpha} \bar{u}_{\alpha,x}\,,\quad F_{2,\alpha}=\bar{c}_{u,\alpha}\bar{\omega}_{\alpha}+2\bar{u}_{\alpha}\bar{u}_{\alpha,x}-2\bar{\psi}_{\alpha} \bar{\omega}_{\alpha,x}\,.$$
Before we perform our energy estimates, we will obtain some basic estimates of the residues.
\begin{lemma}
\label{residue-lemma}
The following estimates hold for $\kappa=\frac{7}{8}<\frac{9}{10}<\alpha<1$.
\begin{enumerate}
    \item Pointwise estimates of the residues:
    $$|\partial_x^i\omega_{\text{res}}|+|\partial_x^i u_{\text{res}}|\lesssim |\alpha-1||\sin x|^{\kappa-i}\,,\quad i=0,1,2,3,$$ $$\|\psi_{\text{res}}\|_{\infty}+\|\psi_{\text{res},x}\|_{\infty}\lesssim |\alpha-1|\,.$$
    \item Refined estimates using cancellations:
    $$|\frac{\alpha-1}{2}\bar{u}_{\alpha,x}-\sin xu_{\text{res},xx}|+|\sin x\partial x[\frac{\alpha-1}{2}\bar{u}_{\alpha,x}-\sin xu_{\text{res},xx}]|\lesssim |\alpha-1|^{1/2}|x||\sin x|^{\alpha-1}\,.$$
\end{enumerate}
\end{lemma}
\begin{proof}
The first part of (1) and (2) can be proved by using direct calculations, and we refer to Lemma 6.1 in \cite{chen2021regularity} for details. There seems to be a typo in (6.11) in \cite{chen2021regularity} where $\bar{\omega}_{\alpha}$ should have been $\bar{\omega}_{\alpha,x}$.
    Furthermore, by the expression in Lemma  \ref{alpha est} we get the second part of (1).
\end{proof}
Similar to the weak advection case, we define the energy $E^2(\tau)=\frac{1}{2}((u_x^2,\rho)+(\omega^2,\rho))\,.$ We will estimate the growth of $E(\tau)$.
The leading order linear estimates $L_1, L_2$ can be obtained in Proposition \ref{prop d}.
The estimates for the nonlinear terms $N_{1,\alpha}, N_{2,\alpha}$ follow almost exactly the same as the weak advection case by using Lemma \ref{lotb}.

\subsection{Nonlinear stability}
By the computation in the previous subsection, we get 
\begin{equation*}
\label{nonlinearalpha}
    \frac{1}{2}\frac{d}{d\tau}E^2(\tau)\leq -0.16E^2+CE^3+((R_{1,\alpha})_x,u_x\rho)+(R_{2,\alpha},\omega\rho)+((F_{1,\alpha})_x,u_x\rho)+(F_{2,\alpha},\omega\rho)\,.
\end{equation*}
Further, we get
$$\begin{aligned}
    ((R_{1,\alpha})_x,u_x\rho)&\leq (-2\psi_{\text{res}}u_{xx},u_x\rho)+(c_u\bar{u}_{\alpha,x}-2u_{\text{res},xx}\psi,u_x\rho)\\&+(\|\omega_{\text{res}}\|_{\infty}+\|u_{\text{res}}\|_{\infty}+C|\alpha-1|)E^2\,.
\end{aligned}$$
For the first term, we can use integration by parts and Lemma  \ref{lotb} to obtain
$$CE^2(\|\psi_{\text{res},x}\|_{\infty}+\|\frac{\psi_{\text{res}}}{\sin x}\|_{\infty})\lesssim \|\psi_{\text{res},x}\|_{\infty}E^2\,.$$
For the second term, we compute $$c_u\bar{u}_{\alpha,x}-2u_{\text{res},xx}\psi=\psi_x(0)[(\alpha-1)\bar{u}_{\alpha,x}-2\sin x u_{\text{res},xx}]+2u_{\text{res},xx}(\sin x\psi_x(0)-\psi)\,.$$
Thus by Lemmas  \ref{lotb}  and \ref{residue-lemma}, we have \begin{equation}
    \label{important}
\|c_u\bar{u}_{\alpha,x}-2u_{\text{res},xx}\psi\|_{\rho}\lesssim E|\alpha-1|^{1/2} +|\alpha-1|\|\frac{\sin x\psi_x(0)-\psi}{|\sin x|x}\|_{\infty}\,.\end{equation}
Finally, for $|x|\geq \pi/2$, we have $$|\frac{\sin x\psi_x(0)-\psi}{|\sin x|x}|\lesssim |\psi_x(0)|+\|\frac{\psi}{\sin x}\|_{\infty}\lesssim E\,.$$
For $|x|< \pi/2$, we use Lemma \ref{alpha est} and $|\sin x|\geq 2/\pi |x|$ to obtain
$$|\frac{\sin x\psi_x(0)-\psi}{|\sin x|x}|\lesssim |\frac{\sin x\psi_x(0)-\psi}{x^2}|\leq |\frac{(\sin x-x)\psi_x(0)}{x^2}|+|\frac{\int_{0}^x(y-x)\omega (y)}{x^2}|\lesssim E\,,$$
where we have used the bound $\|\omega/x\|_{1}\lesssim\|\omega/x\|_{2}\lesssim\|\omega\|_{\rho}$ in the last inequality.
Thus we yield
$$((R_{1,\alpha})_x,u_x\rho)\lesssim |\alpha-1|^{1/2}E^2\,.$$
 Similarly we get $$(R_{2,\alpha},\omega\rho)\leq (-2\frac{\psi}{\sin x}\omega_{\text{res},x}\sin x,\omega\rho)+(c_u\bar{\omega}_{\alpha},\omega\rho)+(2uu_{\text{res},x},\omega\rho)+C|\alpha-1|E^2\,.$$
 For the first two terms, we can estimate them using Lemmas  \ref{lotb}  and \ref{residue-lemma}. For the third term, we use Hardy's inequality to derive
 $$\|uu_{\text{res},x}\|_{\rho}/|\alpha-1|\lesssim\|u/x/\sin x\|_{2}\lesssim \|u/x^2\|_{2}+\|u/(\pi-x)\|_{2}\lesssim\|u_x/x\|_{2}+\|u_x\|_{2}\lesssim E\,.$$
 Therefore we have 
 $$(R_{2,\alpha},\omega\rho)\lesssim|\alpha-1|E^2\,.$$
 
 For the error terms, we can just perform standard norm estimates.  We focus on the pointwise estimates for $x\geq0$ and the case $x<0$ follows by using the odd symmetry of the solution.
 $$
\begin{aligned}
    F_{2,\alpha}&=(\alpha-1)\bar{\psi}_{\alpha,x}(0)\sin^{\alpha} x+(\alpha+1)\cos x\sin^{\alpha} x-2\alpha(\psi_\text{res}+\sin x)\cos x\sin^{\alpha-1} x\\&=(\alpha-1)(\bar{\psi}_{\alpha,x}(0)-\cos x)\sin^{\alpha} x-2\alpha\frac{\psi_{\text{res}}}{\sin x}\cos x\sin^{\alpha} x\,.
\end{aligned}
$$
Therefore, we obtain $$\|F_{2,\alpha}\|_{\rho}\lesssim |\alpha-1|+\|\frac{\psi_{res}}{\sin x}\|_{\infty}\lesssim |\alpha-1|\,.$$
Similarly, we have $$
\begin{aligned}
    (F_{1,\alpha})_x&=\frac{\alpha^2-1}{2}\sin^{\frac{\alpha-1}{2}} x\cos x[\bar{\psi}_{\alpha,x}(0)-\bar{\psi}_{\alpha}\frac{\cos x}{\sin x}]+\sin^{\frac{\alpha+1}{2}}x[(\alpha+1)\bar{\psi}_{\alpha}-2\sin^{\alpha}x]\,.
\end{aligned}
$$
Further, we obtain the following estimate: 
$$\|(\alpha+1)\bar{\psi}_{\alpha}-2\sin^{\alpha}x\|_{\infty}\leq |\alpha-1|\|\bar{\psi}_{\alpha}\|_{\infty}+2\|\psi_\text{res}\|_{\infty}+2\|\sin x-\sin^{\alpha}x\|_{\infty}\lesssim|\alpha-1|\,,$$
$$\begin{aligned}\bar{\psi}_{\alpha,x}(0)-\bar{\psi}_{\alpha}\frac{\cos x}{\sin x}&=(1-\cos x)+\psi_{\text{res},x}(0)-\psi_\text{res}\frac{\cos x}{\sin x}\\&=(1-\cos x)(1+\frac{\psi_\text{res}}{\sin x})+\psi_{\text{res},x}(0)-\frac{\psi_\text{res}}{\sin x}\,.\end{aligned}$$
Combined with the estimate similar to that in \eqref{important}, we get 
$$\|[\bar{\psi}_{\alpha,x}(0)-\bar{\psi}_{\alpha}\frac{\cos x}{\sin x}]\rho^{1/2}\|_{\infty}\lesssim\|\frac{1-\cos x}{x}(1+\frac{\psi_\text{res}}{\sin x})\|_{\infty}+\|[\psi_{\text{res},x}(0)-\frac{\psi_\text{res}}{\sin x}]/x\|_{\infty}\lesssim 1\,.$$
Therefore we yield $$\|(F_{1,\alpha})_x\|_{\rho}\lesssim |\alpha-1|\,.$$
Collecting all the estimates of the residues and the error terms, we arrive at \begin{equation*}
    \frac{d}{d\tau}E(\tau)\leq -(0.16-C|\alpha-1|^{1/2}|)E+CE^2+C|\alpha-1|\,.
\end{equation*}
Similar to Subsection \ref{nonlinearest}, we can perform the bootstrap argument to conclude finite-time blowup.
 \subsection{Estimates in higher-order Sobolev norms and convergence to steady state}
 Following the ideas in Subsection \ref{h2sec}, we can perform estimates in higher-order Sobolev norms and then close the estimates to establish  convergence to a steady state. We use the same energy $K$ and only sketch the main steps here. We first have from Subsection \ref{h2sec} the estimates of $L_i$ and $N_i$ and obtain
$$\begin{aligned}
\frac{1}{2}\frac{d}{d\tau}K^2(\tau)&\leq -K^2+CEK+CEK^2+((R_{1,\alpha})_{xx},\sin^2x u_{xx}\rho)\\&+((R_{2,\alpha})_x,\sin^2x\omega_x \rho)+((F_{1,\alpha})_{xx},\sin^2xu_{xx}\rho)+((F_{2,\alpha})_x,\sin^2x\omega_x\rho)\,.
\end{aligned}$$

By Lemma \ref{residue-lemma}, $\sin x\partial x u_{\text{res}}, \sin x\partial x \omega_{\text{res}}$ shares the same pointwise estimates as $\ u_{\text{res}},  \omega_{\text{res}}$. We can obtain, similar to the estimates in $E$, the estimates $$((R_{1,\alpha})_{xx},\sin^2x u_{xx}\rho)\lesssim|\alpha-1|(E+K)K+K\|[u_{\text{res},xx}(\psi-\sin x\psi_x(0))]_x\sin x\|_{\rho}\,.$$
We further obtain the following estimate $$\begin{aligned}
    \|[u_{\text{res},xx}(\psi-\sin x\psi_x(0))]_x\sin x\|_{\rho}&\lesssim |\alpha-1|E+\|u_{\text{res},xx}(\psi_x-\cos x\psi_x(0))\sin x\|_{\rho}\\&\lesssim |\alpha-1|E+|\alpha-1|\|(\psi_x-\psi_x(0))/x\|_{\infty}\,.
\end{aligned}$$
Finally by the Cauchy-Schwarz inequality,
we have 
\begin{equation}
    \label{cs-al}|(\psi_x-\psi_x(0))/x|\leq\int_{0}^x|\omega(y)/y|dy\lesssim\|\omega/x\|_{2}\lesssim\|\omega\|_{\rho}\lesssim E\,,
\end{equation}
and therefore we conclude $$((R_{1,\alpha})_{xx},\sin^2x u_{xx}\rho)\lesssim|\alpha-1|(E+K)K\,.$$
By similar computations, we have 
$$((R_{2,\alpha})_{x},\sin^2x \omega_{x}\rho)\lesssim|\alpha-1|(E+K)K\,.$$

And for the residue terms, we use similar estimates to obtain
$$\|(F_{2,\alpha})_x\sin x\|_{\rho}\lesssim|\alpha-1|\,,$$
$$  \|(F_{1,\alpha})_{xx}\sin x\|_{\rho}\lesssim|\alpha-1|+|\alpha-1|\|(\frac{\psi_{\text{res}}}{\sin x})_x\rho^{1/2}\sin x\|_{\infty}\;.$$
We can use the triangular inequality to estimate $$|(\frac{\psi_{\text{res}}}{\sin x})_x\rho^{1/2}\sin x|\leq |[\psi_{\text{res},x}(0)-\frac{\psi_\text{res}}{\sin x}]/x|+|(\psi_{\text{res},x}-\psi_{\text{res},x}(0))/x|+|\frac{1-\cos x}{x}\frac{\psi_\text{res}}{\sin x}|\,.$$
Combined with the estimate similar to that in \eqref{important} and \eqref{cs-al}, we conclude $$  \|(F_{1,\alpha})_{xx}\sin x\|_{\rho}\lesssim|\alpha-1|\,.$$

We can finally obtain the same estimate as in Subsection \ref{h2sec} 
$$\frac{d}{d\tau}K(\tau)\leq -(1-C|a-1|)K+CE+C|a-1|+CEK\,.$$
We can again find an absolute constant $\mu_\alpha>1$ such that $$\frac{d}{d\tau}(K+\mu_\alpha E)\leq-(0.1-C|a-1|)(K+\mu_\alpha E)+C|a-1|+C(K+\mu_\alpha E)^2\,.$$
And we can use the bootstrap argument on this higher-order energy $K+\mu_\alpha E$ to conclude a priori estimate in this norm. Now we can perform weighted estimates in time using the energy $J$ as in Subsection \ref{sec ct}, where the linear estimates follow from the linear estimates of $E$, the nonlinear estimates follow from Subsection \ref{sec ct}, and the error term vanishes under time-differentiation. We can obtain the same estimates of $J$ and establish exponential convergence of $J$ to zero. Then we use the argument in \cite{chen2021finite,chen2021slightly} to establish exponential convergence to the steady state and conclude the proof of Theorem \ref{t2}.
 \section{Blowup of the viscous model with weak advection}\label{sec5}
In this section, we follow the strategy of the linear and nonlinear estimates of the weak advection model and establish blowup  for the weak advection viscous model with $a<1$.   Intuitively, the viscosity term is small in the dynamic rescaling formulation and nonlinear stability can be closed using a higher-order norm, where the viscosity term has a stability effect. Therefore we expect that the viscous weak advection model develops a finite time singularity as well. 
\subsection{Dynamic rescaling formulation}
We recall the weak advection  model with viscosity.
\begin{equation}
\label{1dwpv}
\begin{aligned}
u_{ t}+2 a\psi u_{ z} &=2 u \psi_{ z}+\nu u_{zz}\,, \\
\omega_{ t}+2 a\psi\omega_{ z} &=\left(u^{2}\right)_{z}+\nu \omega_{zz}\,, \\
- \psi_{zz} &=\omega\,.
\end{aligned}
\end{equation}
We use  the rescaled variables $$\tilde{u}(x, \tau)=C_{u}(\tau)  u( x, t(\tau))\,,\quad \tilde{\omega}(x, \tau)=C_{u}(\tau)  \omega( x, t(\tau))\,,\quad \tilde{\psi}(x, \tau)=C_{u}(\tau)  \psi(x, t(\tau))\,,$$ where $$
C_{u}(\tau)=C_{u}(0)\exp \left(\int_{0}^{\tau} c_{u}(s) d s\right)\,, \quad t(\tau)=\int_{0}^{\tau} C_{u}(s) d s\,.$$ For solutions to \eqref{1dwpv}, the rescaled variables satisfy the dynamic rescaling equation 
\begin{equation}
\label{1drfv}
\begin{aligned}
\tilde{u}_{\tau}+2a \tilde{\psi}\tilde{u}_{ x} &=2 \tilde{u} \tilde{\psi}_{ x}+c_u \tilde{u}+\nu C_u(\tau)u_{xx}\,, \\
\tilde{\omega}_{ \tau}+2a \tilde{\psi}\tilde{\omega}_{ x} &=\left(\tilde{u}^{2}\right)_{x}+c_{u}\tilde{\omega}+\nu C_u(\tau) \omega_{xx}\,, \\
-\tilde{\psi}_{xx} &=\tilde{\omega}\,.
\end{aligned}
\end{equation}
\begin{remark}
    Different from the rescaling in the inviscid case, we introduce an extra degree of freedom: the constant $C_u(0)$. We will choose it later to ensure that the viscous term has a relatively small scaling compared to the main terms.
\end{remark}
In order to establish a finite time blowup, it suffices to prove the dynamic stability of \eqref{1drfv} with scaling parameter $c_u<-\epsilon<0$ for all time; see also \cite{chen2021finite}.
As before, we will primarily work in the dynamic rescaling formulation and use the notations $\tilde{u}=\bar{u} +u$, where $\bar{u}$ is an approximation steady state and $u$ is the perturbation that we will control in time.  

We consider the following approximate steady state.
$$\bar{\omega}=\bar{u}=\bar{\psi}=\sin x\,,\quad \bar{c}_u(\tau)=2(a-1)\bar{\psi}_x(0)-\nu C_u(\tau)\bar{u}_{xxx}(0)/\bar{u}_{x}(0)=2(a-1)+\nu C_u(\tau)\,.$$
We consider odd perturbations $u$, $\omega$, $\psi$. The odd symmetry of the solution is preserved in time by equation \eqref{1drfv}.  We use the following normalization condition: $c_u=2(a-1){\psi}_x(0)-\nu C_u(\tau)u_{xxx}(0)$. This normalization ensures that $u_x(0)$ remains $0$ if the initial perturbation satisfies $u_x(0,0)=0$. In fact, if  $u_x(\tau,0)=0$, then we obtain
$$\begin{aligned}
    \frac{d}{d\tau}u_x(\tau,0)&=\frac{d}{d\tau}(u_x(\tau,0)+\bar{u}_x(\tau,0))=(2-2a)(u_x(\tau,0)+\bar{u}_x(\tau,0))(\psi_x(\tau,0)+\bar{\psi}_x(\tau,0))\\&+(c_u+\bar{c}_u)(u_x(\tau,0)+\bar{u}_x(\tau,0))+\nu C_u(\tau)(\bar{u}_{xxx}(0)+{u}_{xxx}(0))=0\,.\end{aligned}$$

This particular choice of the approximate steady state and the normalization conditions ensure that $u_x(0,\tau)=0$ for all time provided that the initial perturbation satisfies $u_x(0,0)=0$. We will perform the same weighted norm estimate in the singular weight $\rho$ and the weighted norm $E$ as in the inviscid case. 
 
       \subsection{Estimates of the viscous terms}
    Now the perturbation satisfies 
     \begin{equation}
\label{1drcv}
\begin{aligned}
{u}_{ \tau} &=L_1+(a-1)L'_1+N_1+F_1+\nu C_u(\tau)V^u\,, \\
{\omega}_{ \tau}&=L_2+(a-1)L'_2+N_2+F_2+\nu C_u(\tau)V^\omega\,, \\
-{\psi}_{xx} &={\omega}\,.
\end{aligned}
\end{equation}
Here the terms $V^u$ and $V^\omega$ correspond  to all of the terms containing the effect of the viscosity and we factor out explicitly the small factor $\nu C_u(\tau)$ for a fixed $\nu$. $$V^u=u_{xx}+\bar{u}_{xx}+(1-u_{xxx}(0))(u+\bar{u})=u_{xx}-u_{xxx}(0)\sin x +(1-u_{xxx}(0))u\,,$$$$V^\omega=\omega_{xx}-\omega_{xxx}(0)\sin x +(1-\omega_{xxx}(0))\omega\,.$$
    We invoke the  nonlinear estimates in the inviscid case and obtain for the viscous model:
\begin{equation}
\label{l2-unvis}
    \frac{1}{2}\frac{d}{d\tau}E^2(\tau)\leq -(0.16-C|a-1|)E^2+C|a-1|E+CE^3+\nu C_u(\tau)[((V^u)_x,u_x\rho)+(V^\omega,\omega\rho)]\,.
\end{equation}
    We estimate the viscous terms carefully since they involve singular weights. $$
   ((V^u)_x,u_x\rho)=(u_{xxx}-u_{xxx}(0),u_x{\rho})+u_{xxx}(0)(1-\cos x,u_x{\rho})+(1-u_{xxx}(0))\|u_x\|^2_{\rho}\,.
    $$
    Notice that $$|\rho_x|=\rho \left |\frac{-\sin x}{1-\cos x}\right | \lesssim \rho |\frac{1}{x}|\,, \quad |\rho_{xx}|=\rho \left |\frac{-\cos x}{1-\cos x}+\frac{2\sin^2 x}{(1-\cos x)^2} \right |\lesssim \rho |x^{-2}|$$ are singular near the origin and are smooth elsewhere. We can
use integration by parts twice to compute $$\begin{aligned}(u_{xxx}-u_{xxx}(0),u_x{\rho}&)=-(u_{xx}-xu_{xxx}(0),u_{xx}{\rho})-(u_{xx}-xu_{xxx}(0),\frac{x^2}{2}u_{xxx}(0){\rho_x})\\&-(u_{xx}-xu_{xxx}(0),(u_{x}-\frac{x^2}{2}u_{xxx}(0)){\rho_x})\\&\leq -\|u_{xx}\|^2_{\rho}+C[|u_{xxx}(0)|\|u_{xx}\|_{\rho}+|u_{xxx}(0)|^2+\|\frac{u_x}{x}-\frac{x}{2}u_{xxx}(0)\|_{\rho}^2]\\&\leq -\frac{1}{2}\|u_{xx}\|^2_{\rho}+C|u_{xxx}(0)|^2+C\|\frac{u_x}{x}\|_{\rho}^2\,,\end{aligned}$$
where for  the last inequality we use the weighted AM-GM inequality $ab\leq \epsilon a^2+\frac{1}{4\epsilon}b^2$ for a very small constant $\epsilon$.
Therefore we get \begin{equation}\label{visu-l2}
    ((V^u)_x,u_x\rho)\leq -\frac{1}{2}\|u_{xx}\|^2_{\rho}+C[|u_{xxx}(0)|^2+\|\frac{u_x}{x}\|_{\rho}^2 +(1+|u_{xxx}(0)|)E^2]\,.
\end{equation}
Similarly, we estimate via integration by parts$$\begin{aligned}(\omega_{xx},\omega{\rho})&=-(\omega_{x}-\omega_{x}(0),(\omega_{x}-\omega_{x}(0)){\rho})-(\omega_{x}-\omega_{x}(0),\omega_{x}(0){\rho})\\&-(\omega_{x}-\omega_{x}(0),x\omega_{x}(0){\rho_x})-(\omega_{x}-\omega_{x}(0),(\omega-x\omega_{x}(0)){\rho_x})\\&\leq-\|\omega_{x}-\omega_{x}(0)\|^2_{\rho}+C[|\omega_{x}(0)|\|\frac{\omega_{x}-\omega_{x}(0)}{x}\|_{\rho}+\|\frac{\omega}{x}-\omega_{x}(0)\|_{\rho}^2]\,.\end{aligned}$$
And we obtain
\begin{equation}\label{visw-l2}\begin{aligned}(V^\omega,\omega\rho)
   &\leq -\|\omega_{x}-\omega_{x}(0)\|^2_{\rho}+C[|\omega_{x}(0)|\|\frac{\omega_{x}-\omega_{x}(0)}{x}\|_{\rho}\\&+\|\frac{\omega}{x}-\omega_{x}(0)\|_{\rho}^2+|\omega_{xxx}(0)|^2 +(1+|\omega_{xxx}(0)|)E^2]\,.
   \end{aligned}
\end{equation}
    

   The essential difficulty for the viscous terms is that after integration by parts, the singular weight produces various positive terms, on top of the damping terms $-\|u_{xx}\|_{\rho}$; see \eqref{visu-l2}, \eqref{visw-l2}. Fortunately, the positive terms contribute only to higher-order terms near the origin.
   
   Consider the interval $I=[-\pi/2,\pi/2]$. $\rho$ and $|1/x|$ are upper bounded by a positive constant outside of the interval and we have $$\|\frac{u_x}{x}\|_{\rho}^2\lesssim \|{u_x}\|_{\rho}^2+\|\frac{u_x}{x^2}\|_{L^2(I)}^2\lesssim E^2+\|u_{xxx}\|_{L^\infty(I)}^2\,,$$
   $$\|\frac{\omega}{x}-\omega_{x}(0)\|_{\rho}^2\lesssim \|\omega-x\omega_{x}(0)\|_{\rho}^2+\|\frac{\omega}{x^2}-\omega_{x}(0)/x\|_{L^2(I)}^2\lesssim E^2+|\omega_{x}(0)|^2+\|\omega_{xx}\|_{L^\infty(I)}^2\,,$$
   $$\|\frac{\omega_{x}-\omega_{x}(0)}{x}\|_{\rho}\lesssim\|{\omega_{x}-\omega_{x}(0)}\|_{\rho}+\|\frac{\omega_{x}-\omega_{x}(0)}{x^2}\|_{L^2(I)}\lesssim\|{\omega_{x}-\omega_{x}(0)}\|_{\rho}+\|\omega_{xxx}\|_{L^\infty(I)}\,.$$
   Plugging these estimates into the \eqref{visu-l2}, \eqref{visw-l2} and using again the weighted AM-GM inequality, we can yield 
   \begin{equation}\label{visuw-l2}(V^\omega,\omega\rho)+((V^u)_x,u\rho)
   \leq -\frac{1}{2}(\|\omega_{x}-\omega_{x}(0)\|^2_{\rho}+\|u_{xx}\|^2_{\rho})+C[E_V^2 +(1+E_V)E^2]\,, 
\end{equation}
where $$E_V=\|\omega_{xxx}\|_{L^\infty(I)}+\|u_{xxx}\|_{L^\infty(I)}+|\omega_{x}(0)|\,.$$

   \subsection{Estimates in a higher-order norm}
    We will use a weighted higher-order norm to close the estimates. To have a good estimate in this higher-order norm, it needs to satisfy three criteria. First, we need to extract damping in the leading order linear term. Secondly, we need to bound the terms like $\omega_{xxx}(0)$ using interpolation between the lower and the higher-order norms via the Gagliardo-Nirenberg inequality; therefore it needs to be at least as strong as a regular higher-order norm near the origin. Thirdly, we need damping for the diffusion terms to close the estimates. This motivates us to choose a combination of the $k-$th order weighted norms for $k\geq1$:
    $$E_{k}^2(\tau)=(u^{(k+1)},u^{(k+1)}\rho_k)+(\omega^{(k)},\omega^{(k)}\rho_k)\,, \quad \rho_k=(1+\cos x)^k\;,$$
    where we use the notation that $f^{(k)}=\partial_x^k f$. We denote $E_0=E$ and $\rho_0=\rho$.
    \begin{remark}
    \label{remark rm}
        This weighted norm immediately satisfies criterion 2 and we will verify in the linear estimates that it satisfies criterion 1. Finally, a clever combination of the weighted norms can produce damping for the viscous terms and we make the damping terms in the estimates of the $(k-1)$-th order norms greater than the positive terms in the estimates of the $k$-th order norms. We will elaborate on those points and establish the nonlinear estimates.
    \end{remark}

    Now we can estimate $\frac{d}{d\tau}E_{k}(\tau)$ for $k>0$ as follows
    $$\begin{aligned}
    &\frac{1}{2}\frac{d}{d\tau}E_{k}^2(\tau)=(L_2^{(k)}+(a-1)(L_2')^{(k)}+N_2^{(k)}+F_2^{(k)}+\nu C_u(t)(V^\omega)^{(k)},\omega^{(k)}\rho_k)\\&+(L_1^{(k+1)}+(a-1)(L_1')^{(k+1)}+N_1^{(k+1)}+F_1^{(k+1)}+\nu C_u(t)(V^u)^{(k+1)},u^{(k+1)}\rho_k)\,,
    \end{aligned}$$ 
    where the parts $L_i,L_i',F_i,N_i$ are defined exactly the same as in the inviscid case. 

    We first look at the viscous terms.
     We have for example
    $$((V^u)^{(k+1)},u^{(k+1)}(1+\cos x)^k)\leq (u^{(k+3)},u^{(k+1)}(1+\cos x)^k)+CE_VE_k+(1+E_V)E_k^2\,.$$
    We use integration by parts twice to obtain
    $$\begin{aligned}
        &(u^{(k+3)},u^{(k+1)}(1+\cos x)^k)=-(u^{(k+2)},u^{(k+2)}(1+\cos x)^k-u^{(k+1)}k(1+\cos x)^{k-1}\sin x)\\&=-(u^{(k+2)},u^{(k+2)}(1+\cos x)^k)+\frac{1}{2}(u^{(k+1)},u^{(k+1)}k(1+\cos x)^{k-1}[(k-1)-k\cos x])\\&\leq -(u^{(k+2)},u^{(k+2)}(1+\cos x)^k)+C(k)(u^{(k+1)},u^{(k+1)}(1+\cos x)^{k-1})\,.
    \end{aligned}$$
    We can also get a similar bound for $V^\omega$. Therefore combined with the leading order estimate \eqref{visuw-l2} and using the idea in Remark \ref{remark rm}, we conclude that for small enough constants $0<\mu<\mu_0({k_0})<1$, we have the following viscous estimate \begin{equation}\label{viscous-final}
        \sum_{k=0}^{k_0}\mu^k[((V^u)^{(k+1)},u^{(k+1)}\rho_k)+((V^\omega)^{(k)},\omega^{(k)}\rho_k)]\leq C(k)[E_V^2+(1+E_V)\sum_{k=0}^{k_0}\mu^kE_k^2]\,. 
    \end{equation}
    Here $\mu_0(k_0)$ is a generic constant depending on $k_0$. We can choose $k_0$ large enough later so that $E_V$ can be bounded using the interpolation inequalities.
    
    Now we look at the linear terms and extract damping. We denote the terms as lower order terms (l.o.t. for short) if their $\rho_k$-weighted $L^2$-norms are bounded by $\sum_{i=0}^{k-1}E_i$. For the terms of intermediate order, since  $\rho_k\leq C(k)\rho_i$ for $i<k$, combined with the classical elliptic estimate,  we can show that $u^j$, $\psi^i$ for $0\leq j<k+1$ and $0\leq i<k+2$ are l.o.t.
    Using the l.o.t. notation, we keep track only of the higher-order terms $$(L_1^{(k+1)},u^{(k+1)}\rho_k)=(-2\sin x u^{(k+2)}-2k\cos x u^{(k+1)}+2\sin x \psi^{(k+2)}+\text{l.o.t.},u^{(k+1)}\rho_k)\,,$$
    $$(L_2^{(k)},\omega^{(k)}\rho_k)=(-2\sin x \omega^{(k+1)}-2k\cos x \omega^{(k)}+2\sin x u^{(k+1)}+\text{l.o.t.},\omega^{(k)}\rho_k)\,.$$
    Again we have a crucial cancellation of the cross terms and for the leading order terms we use integration by parts to obtain for example $$\begin{aligned}(-2\sin x u^{(k+2)}-2k\cos x u^{(k+1)},u^{(k+1)}\rho_k)&=(u^{(k+1)},u^{(k+1)}(-k-(k-1)\cos x)\rho_k)\\&\leq -(u^{(k+1)},u^{(k+1)}\rho_k)\,.\end{aligned}$$
    Therefore we derive the following estimate $$(L_1^{(k+1)},u^{(k+1)}\rho_k)+(L_2^{(k)},\omega^{(k)}\rho_k)\leq -E_k^2+C(k)\sum_{i=0}^{k-1}E_i E_k\leq -\frac{1}{2}E_k^2+C(k)\sum_{i=0}^{k-1}E_i^2\,.$$
    
    Similarly we have $$(L_1'^{(k+1)},u^{(k+1)}\rho_k)+(L_2'^{(k)},\omega^{(k)}\rho_k)\leq 2E_k^2+C(k)\sum_{i=0}^{k-1}E_i^2\,.$$
    
    We have the trivial bound for the error term 
    $$(F_1^{(k+1)},u^{(k+1)}\rho_k)+(F_2^{(k)},\omega^{(k)}\rho_k)\leq C(k) (a-1)E_k\,.$$
    
    The nonlinear terms are more subtle. We will show that \begin{equation}\label{nonlinea-f}
        (N_1^{(k+1)},u^{(k+1)}\rho_k)\leq  C(k)\sum_{i=0}^{k}E_i^2E_k\,,
    \end{equation}
    and we can have the same bound for $\omega$.
    In fact, for a canonical term in $N_1^{(k+1)}$ and $N_2^{(k)}$, it is of the form $\psi^{(i)}u^{(k+2-i)}$ or $\psi^{(i)}\psi^{(k+3-i)}$ or $u^{(i)}u^{(k+1-i)}.$ For the terms $\psi u^{(k+2)}$ and $\psi \omega^{(k+1)}$, we can use integration by parts and Lemma \ref{lotb} to show that  $$(\psi u^{(k+2)},u^{(k+1)}(1+\cos x)^k)\leq C(k)(\|\psi_x\|_{\infty}+\|\frac{\psi}{\sin x}\|_{\infty})E_k^2\leq C(k)E E_k^2\,.$$
    The terms associated with $\psi_x u^{(k+1)}$, $\psi_x \omega^{(k)}$, $u u^{(k+1)}$, and $u \omega^{(k)}$  have the same bound trivially. We can then focus on controlling the weighted norms of $\omega^{(i)}u^{(k-i)}$, $\omega^{(i)}\omega^{(k-1-i)}$, $u^{(i+1)}u^{(k-i)}$ for indices $0<i<k$ to establish the bound \eqref{nonlinea-f}. For example, we get
    $$\|\omega^{(i)}u^{(k-i)}(1+\cos x)^{k/2}\|_{2}\leq \|\omega^{(i)}(1+\cos x)^{(i+1)/2}\|_{\infty}E_{k-1-i}\,.$$
    Finally, by the fundamental theorem of calculus, we can bound the $L^{\infty}$-norm by  $$\begin{aligned}
        & C(K)[\|\omega^{(i+1)}(1+\cos x)^{(i+1)/2}\|_{1}+\|\omega^{(i)}(1+\cos x)^{(i-1)/2}\sin x\|_{1}]\\&\leq C(K)[E_{i+1}+\|\omega^{(i)}(1+\cos x)^{(i)/2}\|_{1}]\leq C(k)\sum_{i=0}^{k}E_i\,.
    \end{aligned}
        $$
        Therefore we conclude that \eqref{nonlinea-f} holds.
\subsection{Collection of norms and finite time blowup}
We collect the bounds \eqref{viscous-final} \eqref{nonlinea-f} for viscous and nonlinear terms, along with the linear bounds and the leading order estimate \eqref{l2-unvis}. For any fixed ${k_0}$, there exists a small enough constant $0<\mu_1({k_0})<\mu_0({k_0})$, such that the following estimate holds
$$\frac{d}{d\tau}I_{k_0}^2\leq -(0.1-C|a-1|)I_{k_0}^2+C|a-1|I_{k_0}+CI_{k_0}^3+C\nu C_u(\tau)[E_V^2+(1+E_V)I_{k_0}^2]\,,$$
where the energy is defined as $$I_{k_0}^2=\sum_{k=0}^{k_0}\mu_1^k({k_0})E_k^2\,.$$
Here the constants depend on $k_0$ and $\mu$ but once we first prescribe $k_0$ then $\mu=\mu_1({k_0})$, they become just constants. We will later make our $C_u(\tau)$ and $|a-1|$ small to close the argument.

    Finally, by the Gagliardo-Nirenberg inequality, for $k=1,3$, we have $$\|\omega^{(k)}\|_{L^{\infty}(I)}\lesssim \|\omega^{(4)}\|^{\theta}_{L^{2}(I)}\|\omega\|^{1-\theta}_{L^{2}(I)}\,, \quad \theta=\frac{k+1/2}{4}\,.$$
    This is the classical Gagliardo-Nirenberg inequality applied to a bounded domain, and we can just use the extension technique to prove it; see for example \cite{li2021note}. We get similar bounds involving $u$ and  conclude that $E_V\lesssim I_{k_0}$, for any fixed $k_0\geq 4$. For example, we just take $k_0=4$ and obtain
    $$\frac{d}{d\tau}I_{4}\leq -(0.1-C|a-1|)I_{4}+C|a-1|+CI_{4}^2+C\nu C_u(\tau)(1+I_{4})I_{4}\,.$$

Now we choose $C_u(0)=|a-1|^2$ for $|a-1|<\delta$ with a small enough $\delta>0$. It is easy to check that the bootstrap argument for $I_{4}\leq C|a-1|$ and $C_u(\tau)\leq C_u(0)\exp((a-1)t)\leq C_u(0)$ will hold for all time provided that it holds initially. We again use the estimate for the normalization constants $$c_u+\bar{c}_u=2(a-1)+\nu C_u(t)(1-u_{xxx}(0))+2(a-1)\psi_x(0)<(a-1)<0\,.$$
    Thus we can obtain a blowup in finite time in the physical variables. 
\section{Appendix}
\label{sec app}
\begin{lemma}
\label{cap}
    Assume $\sum_{k\geq 1} a_k^2+c_k^2<\infty$, then we have the following inequality for \eqref{dample}
    $$\begin{aligned}
 F(a,c)&\coloneqq \sum_{k\geq1}\{ a_k^2(0.84+\frac{1}{k^2}-\frac{1}{(k-1)^2})+c_k^2(0.84+\frac{1}{k(k+1)})+ 2a_k a_{k+1}\frac{1}{(k+1)^2}\\&+2 a_k \sum_{j>k+1}a_j(\frac{1}{j^2}-\frac{1}{(j-1)^2})+2a_kc_k\frac{1+2k-k^2}{2k^2(k+1)}+2a_{k+1}c_k\frac{k^2-k-1}{2k^2(k+1)^2}\\&-2a_{k+2}c_k\frac{k+2}{2(k+1)^2}+\sum_{j>k}2a_k c_j\frac{1}{j(j+1)}\}\geq0\,.\end{aligned}$$
\end{lemma}
\begin{proof}
    Denote the summation of terms in  $F(a,c)$ that only involve $a_i$, $c_j$ for $i,j\leq N$ as $F_N(a,c)$. Here $N=200$. This quadratic form $F_N(a,c)$ can be expressed as $a^{(N),T}F^{(N)}c^{(N)}$, where $a^{(N)}, c^{(N)}$ are two vectors with entries $a_i$,$c_j$ respectively and $F^{(N)}$ is a symmetric matrix. We numerically verify using interval arithmetic in Matlab that the smallest eigenvalue of $F^{(N)}$ is greater than $0.01$; see remarks after the proof for details. Therefore we have 
    $$F_N(a,c)\geq0.01\sum_{k=1 }^N (a_k^2+c_k^2)\,.$$

For the remainder $F(a,c)-F_N(a,c)$, we estimate it term by term via the trivial bound $2ab\geq-(a^2+b^2)$ and obtain
    $$\begin{aligned}&F(a,c)-F_N(a,c)\geq\sum_{k>N}[ a_k^2(0.84+\frac{1}{k^2}-\frac{1}{(k-1)^2})+c_k^2(0.84+\frac{1}{k(k+1)})] \\&+\sum_{k=1}^{ N} a_k^2(-\frac{2}{N^2}-\frac{1}{N+1}) +\sum_{k=N-1}^{ N} c_k^2(-\frac{N^2-N-1}{2N^2(N+1)^2}-\frac{N+1}{2N^2}) \\&+\sum_{k> N} a_k^2(-\frac{3}{N^2}-N(\frac{1}{(N)^2}-\frac{1}{(N+1)^2})-\frac{N^2-N-1}{2N^2(N+1)^2}-\frac{N+1}{2N^2}-\frac{N^2-2N-1}{2N^2(N+1)}\\&-\frac{1}{N+2})+ \sum_{k> N} c_k^2(-\frac{N^2-N-1}{2N^2(N+1)^2}-\frac{N+1}{2N^2}-\frac{N^2-2N-1}{2N^2(N+1)}-\frac{1}{N+2})\,. \end{aligned}$$
For $N=200$, we estimate all of the coefficients by a lower bound and obtain
    $$F(a,c)-F_N(a,c)\geq -\frac{2}{N}\sum_{k=1}^{ N} (a_k^2+c_k^2)+(0.84-\frac{3}{N})\sum_{k> N} (a_k^2+c_k^2)\geq-\frac{2}{N}\sum_{k=1}^{ N} (a_k^2+c_k^2)\,.$$
    Therefore we conclude $F(a,c)\geq 0$.
\end{proof}
\begin{remark}
    We now explain how to verify that the smallest eigenvalue of the symmetric matrix $F^{(200)}$ is greater than $0.01$. We proceed in three steps.
    \begin{enumerate}
        \item We first use Matlab to perform an (approximate) SVD decomposition of $$F^{(200)}-0.011I\approx VDV'\,.$$
        Here $D$ is the diagonal matrix consisting of (approximate) eigenvalues of $F^{(200)}-0.011I$, and $V$ is the unitary matrix consisting of (approximate) eigenvectors of $F^{(200)}-0.011I$.
        \item We use interval arithmetic to verify that the maximal absolute value of entries of $F^{(200)}-0.011I- VDV'$ is at most $10^{-10}$. Therefore the spectral norm of $F^{(200)}-0.011I- VDV'$, which is bounded by its $1$-norm, is rigorously bounded from above by $200 \times 10^{-10}$.
        \item Since $D$ has positive entries, we know that $VDV'$ is positive definite. We  conclude that  $$F^{(200)}-0.01I=VDV'+0.001I+F^{(200)}-0.011I- VDV'$$ is positive definite.
    \end{enumerate}
\end{remark}

\vspace{0.2in}
{\bf Acknowledgments.} The research was in part supported by NSF Grant DMS-2205590 and the Choi Family Gift Fund and the Graduate Fellowship Fund. We would like to thank Dr. Jiajie Chen for some stimulating discussions, especially on the design of the energy norm. {We also thank Changhe Yang for some valuable discussions in the early stage of this project.}

\bibliographystyle{abbrv}
\bibliography{ref}
\end{document}